\documentclass[reqno,a4paper]{amsart}
\usepackage{amssymb}
\usepackage{amsmath}
\usepackage{amsfonts}

\newtheorem{theorem}{Theorem}[section]
\theoremstyle{plain}

\newtheorem{corollary}{Corollary}[section]

\newtheorem{definition}{Definition}[section]
\newtheorem{example}{Example}[section]

\newtheorem{lemma}{Lemma}[section]

\numberwithin{equation}{section}
\input{tcilatex}

\begin{document}
\title[]{$f$-EIKONAL HELIX SUBMANIFOLDS AND $f$-EIKONAL HELIX CURVES}
\author{Evren Z\i plar}
\address{Department of Mathematics, Faculty of Science, University of
Ankara, Tando\u{g}an, Turkey}
\email{evrenziplar@yahoo.com}
\urladdr{}
\author{Ali \c{S}enol}
\address{Department of Mathematics, Faculty of Science, \c{C}ank\i r\i\
Karatekin University, \c{C}ank\i r\i , Turkey}
\email{asenol@karatekin.edu.tr}
\author{Yusuf Yayl\i }
\address{Department of Mathematics, Faculty of Science, University of
Ankara, Tando\u{g}an, Turkey}
\email{yayli@science.ankara.edu.tr}
\thanks{}
\urladdr{}
\date{}
\keywords{Helix submanifold; Eikonal function; Helix line.\\
Corresponding author: Evren Z\i plar, e-mail: evrenziplar@yahoo.com}
\thanks{}

\begin{abstract}
Let $M\subset 
\mathbb{R}
^{n}$ be a Riemannian helix submanifold with respect to the unit direction $%
d\in 
\mathbb{R}
^{n}$ and $f:M\rightarrow 
\mathbb{R}
$ be a eikonal function. We say that $M$ is a $f$-eikonal helix submanifold
if for each $q\in M$ the angle between $\nabla f$ and $d$ is constant.Let $%
M\subset 
\mathbb{R}
^{n}$ be a Riemannian submanifold and $\alpha :I\rightarrow M$ be a curve
with unit tangent $T$. Let $f:M\rightarrow 
\mathbb{R}
$ be a eikonal function along the curve $\alpha $. We say that $\alpha $ is
a $f$-eikonal helix curve if the angle between $\nabla f$ and $T$ \ is
constant along the curve $\alpha $. $\nabla f$ will be called as the axis of
the $f$-eikonal helix curve.The aim of this article is to give that the
relations between $f$-eikonal helix submanifolds and $f$-eikonal helix
curves, and to investigate $f$-eikonal helix curves on Riemannian manifolds.
\end{abstract}

\maketitle

\section{Introduction}

In differential geometry of manifolds, an helix submanifold of $%
\mathbb{R}
^{n}$ with respect to a fixed direction $d$ in $%
\mathbb{R}
^{n}$ is defined by the property that tangent space makes a constant angle
with the fixed direction $d$ (helix direction) in [3]. Di Scala and Ruiz-Hern%
\'{a}ndez have introduced the concept of these manifolds in [3].

Recently, M. Ghomi worked out the shadow problem given by H.Wente. And, He
mentioned the shadow boundary in [6]. Ruiz-Hern\'{a}ndez investigated that
shadow boundaries are related to helix submanifolds in [10].

Helix hypersurfaces have been worked in nonflat ambient spaces in [4,5].
Cermelli and Di Scala have also studied helix hypersurfaces in liquid
cristals in [2].

The plan of this article is as follows. Section 2, we give some important
definitions and remarks which will be used in other sections.In section 3,
we define $f$-eikonal helix submanifolds and define $f$-eikonal helix
curves. And also, we give an important property between $f$-eikonal helix
submanifolds and $f$-eikonal helix curves, see Theorem 3.2. In Theorem 3.1
and 3.3, we show that when a curve on a manifold is $f$-eikonal helix curve.
Besides,we give the important relation between geodesic curves and $f$%
-eikonal helix curves, see Theorem 3.4. Section 4, in 3-dimensional
Riemannian manifold, we find out the axis of a $f$-eikonal helix curve and
we give the relation between the curvatures of the curve in Theorem 4.1.
Then, we give more important corollary relating to helix submanifolds. In
section 5, we specify the relation between $f$-eikonal helix curve and
general helix.

\section{Basic Definitions}

\begin{definition}
Given a submanifold $M\subset 
\mathbb{R}
^{n}$ and an unitary vector $d$ in $%
\mathbb{R}
^{n}$, we say that $M$ is a helix with respect to $d$ if for each $q\in M$
the angle between $d$ and $T_{q}M$ is constant.

Let us recall that a unitary vector $d$ can be decomposed in its tangent and
orthogonal components along the submanifold $M$, i.e. $d=\cos (\theta
)T^{\ast }+\sin (\theta )\xi $ with $\left\Vert T^{\ast }\right\Vert
=\left\Vert \xi \right\Vert =1$, where $T^{\ast }\in TM$ and $\xi \in
\vartheta (M)$.The angle between $d$ and $T_{q}M$ is constant if and only if
the tangential component of $d$ has constant length $\left\Vert \cos (\theta
)T^{\ast }\right\Vert =\cos (\theta )$. We can assume that $0<\theta <\frac{%
\pi }{2}$ and we can say that $M$ is a helix of angle $\theta $.

We will call $T^{\ast }$ and $\xi $ the tangent and normal directions of the
helix submanifold $M$. We can call $d$ the helix direction of $M$ and we
will assume $d$ always to be unitary [3].
\end{definition}

\begin{definition}
Let $M\subset 
\mathbb{R}
^{n}$ be a helix submanifold of angle $\theta \neq \frac{\pi }{2}$ w.r. to
the direction $d\in 
\mathbb{R}
^{n}$. We will call the integral curves of the tangent direction $T^{\ast }$
of the helix $M$, the helix lines of $M$ w.r.to $d$ [3].
\end{definition}

\noindent \textbf{Remark 2.1 }\textit{We say that }$\xi $\textit{\ is
parallel normal in the direction }$X\in TM$\textit{\ if }$\ \nabla
_{X}^{\perp }\xi =0$\textit{. Here, }$\nabla ^{\perp }$\textit{\ denotes the
normal connection of }$M$\textit{\ induced by the standard covariant
derivative of the Euclidean ambient. And, we denote by }$D$\textit{\ the
standard covariant derivative in }$%
\mathbb{R}
^{n}$\textit{\ and denote by }$\nabla $\textit{\ the induced covariant
derivative in }$M$\textit{. \ [3]. }

\begin{definition}
Let $M$ be a submanifold of the Riemannian manifold $%
\mathbb{R}
^{n}$ and let $D$ be the Riemannian connexion on $%
\mathbb{R}
^{n}$. For $C^{\infty \text{ }}$fields $X$ and $Y$ with domain $A$ on $M$
(and tangent to $M$), define $\nabla _{X}Y$ and $V(X,Y)$ on $A$ by
decomposing $D_{X}Y$ into unique tangential and normal components,
respectively; thus,%
\begin{equation*}
D_{X}Y=\nabla _{X}Y+V(X,Y)\text{. }
\end{equation*}%
Then, $\nabla $ is the Riemannian connexion on $M$ and $V$ is a symmetric
vector-valued 2-covariant $C^{\infty \text{ }}$tensor called the second
fundamental tensor. The above composition equation is called the Gauss
equation [7].
\end{definition}

\noindent \textbf{Remark 2.2 }\textit{Let us observe that for any helix
euclidean submanifold }$M$\textit{, the following system holds for every }$%
X\in TM$\textit{, where the helix direction }$d=\cos (\theta )T^{\ast }+\sin
(\theta )\xi $\textit{.\medskip \vspace{0in} }%
\begin{equation}
\cos (\theta )\nabla _{X}T^{\ast }-\sin (\theta )A^{\xi }(X)=0
\end{equation}%
\begin{equation}
\cos (\theta )V(X,T^{\ast })+\sin (\theta )\nabla _{X}^{\bot }\xi =0
\end{equation}%
\textit{[3].}

\begin{definition}
Let $(M,g)$ be a Riemannian manifold, where $g$ is the metric. Let $%
f:M\rightarrow 
\mathbb{R}
$ be a function and let $\nabla f$ be its gradient, i.e., $df(X)=g(\nabla f$ 
$,X)$. We say that $f$ is eikonal if it satisfies:%
\begin{equation*}
\left\Vert \nabla f\right\Vert =\text{constant.}
\end{equation*}%
[3].
\end{definition}

\begin{definition}
Let $\alpha =\alpha (t):I\subset 
\mathbb{R}
\rightarrow M$ be an immersed curve in 3-dimensional Riemannian manifold $M$
. The unit tangent vector field of $\alpha $ will be denoted by $T$. Also, $%
\kappa >0$ and $\tau $ will denote the curvature and torsion of $\alpha $,
respectively.Therefore if $\left\{ T,N,B\right\} $ is the Frenet frame of $%
\alpha $ and $\nabla $ is the Levi-Civita connection of $M$, then one can
write the Frenet equations of $\alpha $ as%
\begin{equation*}
\nabla _{T}T=\kappa N
\end{equation*}%
\begin{equation*}
\nabla _{T}N=-\kappa T+\tau B
\end{equation*}%
\begin{equation*}
\nabla _{T}B=-\tau N
\end{equation*}%
[1].
\end{definition}

\begin{definition}
Let $\alpha :I\subset 
\mathbb{R}
\rightarrow E^{n}$ be a curve in $E^{n}$ with arc-length parameter $s$ and
let $X$ be a unit constant vector of $E^{n}$. For all $s\in I$, if%
\begin{equation*}
\left\langle V_{1},X\right\rangle =\cos (\varphi ),\varphi \neq \frac{\pi }{2%
},\varphi =\text{constant,}
\end{equation*}%
then the curve $\alpha $ is called a general helix in $E^{n}$, where $V_{1}$
is the unit tangent vector of $\alpha $ at its point $\alpha (s)$ and $%
\varphi $ is a constant angle between the vector fields $V_{1}$ and $X$ [12].
\end{definition}

Throughout all section, the submanifolds $M\subset 
\mathbb{R}
^{n}$ have the induced metric by $%
\mathbb{R}
^{n}$.

\section{$f$-EIKONAL HELIX CURVES}

In this section, we define $f$-eikonal helix submanifolds and define $f$%
-eikonal helix curves. And also, we give an important property between $f$%
-eikonal helix submanifolds and $f$-eikonal helix curves, see Theorem 3.2.
In Theorem 3.1 and 3.3, we show that when a curve on a manifold is $f$%
-eikonal helix curve. Besides,we give the important relation between
geodesic curves and $f$-eikonal helix curves, see Theorem 3.4.

\begin{definition}
Let $M\subset 
\mathbb{R}
^{n}$ be a Riemannian helix submanifold with respect to the unit direction $%
d\in 
\mathbb{R}
^{n}$ and $f:M\rightarrow 
\mathbb{R}
$ be a eikonal function. We say that $M$ is a $f$-eikonal helix submanifold
if for each $q\in M$ the angle between $\nabla f$ and $d$ is constant.
\end{definition}

For definition 3.1, $\left\langle \nabla f,d\right\rangle =$ constant since $%
\left\Vert \nabla f\right\Vert $ and $d$ are constant.

\begin{example}
Let $M\subset 
\mathbb{R}
^{n}$ be a Riemannian helix submanifold with respect to the unit direction $%
d\in 
\mathbb{R}
^{n}$. Let us assume that the tangent component of $d$ equals $\nabla f$ for
a eikonal function $f:M\rightarrow 
\mathbb{R}
$. Because of the definition helix submanifold, we have $\left\langle \nabla
f,d\right\rangle =$ constant. That is, $M$ is a $f$-eikonal helix
submanifold.
\end{example}

\begin{definition}
Let $M\subset 
\mathbb{R}
^{n}$ be a Riemannian submanifold and $\alpha :I\rightarrow M$ be a curve
with unit tangent $T$. Let $f:M\rightarrow 
\mathbb{R}
$ be a eikonal function along the curve $\alpha $, i.e. $\left\Vert \nabla
f\right\Vert =$constant along the curve $\alpha .$We say that $\alpha $ is a 
$f$-eikonal helix curve if the angle between $\nabla f$ and $T$ \ is
constant along the curve $\alpha $. $\nabla f$ will be called as the axis of
the $f$-eikonal helix curve.
\end{definition}

\begin{example}
Let $M\subset 
\mathbb{R}
^{n}$ be a Riemannian submanifold and $\alpha :I\rightarrow M$ be a curve
with unit tangent $T$. Let $f:M\rightarrow 
\mathbb{R}
$ be a eikonal function along the curve $\alpha $. If $\nabla f$ \ equals $T$%
, then $\left\langle \nabla f\ ,\nabla f\ \right\rangle =$constant. That is, 
$\alpha $ is a $f$-eikonal helix curve.
\end{example}

\begin{example}
We consider the Riemannian manifold $M=%
\mathbb{R}
^{3}$. Let%
\begin{equation*}
f:M\rightarrow 
\mathbb{R}%
\end{equation*}%
\begin{equation*}
\left( x,y,z\right) \rightarrow f\left( x,y,z\right) =x^{2}+y^{2}+z
\end{equation*}%
be a function defined on $M$. Then, the curve%
\begin{equation*}
\alpha :I\subset 
\mathbb{R}
\rightarrow M
\end{equation*}%
\begin{equation*}
s\rightarrow \alpha \left( s\right) =(\cos \frac{s}{\sqrt{2}},\sin \frac{s}{%
\sqrt{2}},\frac{s}{\sqrt{2}})
\end{equation*}%
is a $f$-eikonal helix curve on $M$.

Firstly, we will show that $f$ is a eikonal function along the curve $\alpha 
$. If we compute $\nabla f$, we find $\nabla f$ as%
\begin{equation*}
\nabla f=\left( 2x,2y,1\right) \text{.}
\end{equation*}%
So, we get%
\begin{equation*}
\left\Vert \nabla f\right\Vert =\sqrt{4\left( x^{2}+y^{2}\right) +1}\text{.}
\end{equation*}%
And, if we compute $\left\Vert \nabla f\right\Vert $ along the curve $\alpha 
$, we find%
\begin{equation*}
\left\Vert \nabla f\right\Vert |_{\alpha }=\sqrt{5}=\text{constant.}
\end{equation*}%
That is, $f$ is a eikonal function along the curve $\alpha $.

Now, we will show that the angle between $\nabla f$ and $T$ (the unit
tangent of $\alpha $) is constant along the curve $\alpha $. Since%
\begin{equation*}
\nabla f|_{\alpha }=(2\cos \frac{s}{\sqrt{2}},2\sin \frac{s}{\sqrt{2}},1)
\end{equation*}%
and%
\begin{equation*}
T=(-\frac{1}{\sqrt{2}}\sin \frac{s}{\sqrt{2}},\frac{1}{\sqrt{2}}\cos \frac{s%
}{\sqrt{2}},\frac{1}{\sqrt{2}})\text{ ,}
\end{equation*}%
we obtain 
\begin{equation*}
\left\langle \nabla f|_{\alpha },T\right\rangle =\frac{1}{\sqrt{2}}=\cos
(\theta )
\end{equation*}%
along the curve $\alpha $, where $\theta $ is the angle between $\nabla f$
and $T$. Consequently, $\alpha $ is a $f$-eikonal helix curve on $M$.
\end{example}

\begin{example}
We consider the Riemannian manifold $M=%
\mathbb{R}
^{n}$. Let%
\begin{equation*}
f:M\rightarrow 
\mathbb{R}%
\end{equation*}%
\begin{equation*}
\left( x_{1},x_{2},...,x_{n}\right) \rightarrow f\left(
x_{1},x_{2},...,x_{n}\right) =a_{1}x_{1}+...+a_{n}x_{n}+c
\end{equation*}%
be a function defined on $M$, where $a_{1},...,a_{n},c$ are constant. Then,
all generalized helices with the axis $\left( a_{1},a_{2},...,a_{n}\right) $
are $f$-eikonal helices.

Let $\alpha $ be any generalized helice with the axis $X=\left(
a_{1},a_{2},...,a_{n}\right) $. Then, the angle between the unit tangent of $%
\alpha $ and $X$ is constant along the curve $\alpha $. On the other hand,
since $\nabla f=\left( a_{1},a_{2},...,a_{n}\right) $, we have $X=\nabla f$.
So, we can easily say that the angle between the unit tangent of $\alpha $
and $\nabla f$ is constant along the curve $\alpha $. Also, $\left\Vert
\nabla f\right\Vert =$constant. Consequently, the curve $\alpha $ is a $f$%
-eikonal helix. Since $\alpha $ is arbitrary, all generalized helices with
the axis $\left( a_{1},a_{2},...,a_{n}\right) $ are $f$-eikonal helices.
\end{example}

\begin{theorem}
Let $M\subset 
\mathbb{R}
^{n}$ be a Riemannian submanifold and $\alpha :I\rightarrow M$ be a curve
with unit tangent $T$. Let $f:M\rightarrow 
\mathbb{R}
$ be a eikonal function along the curve $\alpha $. Then, $\alpha $ is a $f$%
-eikonal helix curve if and only if $f$ is a linear function along the curve 
$\alpha $.
\end{theorem}

\begin{proof}
Firstly, we assume that $\alpha $ is a $f$-eikonal helix curve. Since $f$ is
a eikonal function along the curve $\alpha $, $\left\Vert \nabla
f\right\Vert |_{\alpha }=$constant. On the other hand, we know that $%
X[f]=\left\langle \nabla f,X\right\rangle $ for each $X\in TM$ (see
Definition 2.4). In particular, for $X=T$, we have 
\begin{eqnarray*}
\left\langle \nabla f|_{\alpha },T\right\rangle &=&T[f] \\
&=&\frac{d}{ds}\left( f\circ \alpha \right) \text{ .}
\end{eqnarray*}%
And, since $\left\Vert \nabla f\right\Vert |_{\alpha }=$constant and $\alpha 
$ is a $f$-eikonal, $\left\langle \nabla f|_{\alpha },T\right\rangle = $%
constant. Thus, we obtain%
\begin{equation*}
\frac{d}{ds}\left( f\circ \alpha \right) =\text{constant .}
\end{equation*}%
In other words, $f|_{\alpha }$ is a linear function.

Conversely, we assume that $f|_{\alpha }$ is a linear function. Clearly, 
\begin{equation*}
\frac{d}{ds}\left( f\circ \alpha \right) =\text{constant .}
\end{equation*}%
Hence, we get%
\begin{equation*}
\left\langle \nabla f|_{\alpha },T\right\rangle =\text{constant .}
\end{equation*}%
And, since $\left\Vert \nabla f\right\Vert |_{\alpha }=$constant and $T$ is
unit, the angle between $\nabla f$ and $T$ is constant along the curve $%
\alpha $. That is, $\alpha $ is a $f$-eikonal helix curve.
\end{proof}

\noindent This completes the proof of the Theorem.

\begin{theorem}
Let $M\subset 
\mathbb{R}
^{n}$ be a $f$-eikonal helix submanifold .Then, the helix lines of $M$ are $%
f $-eikonal helix curves.
\end{theorem}

\begin{proof}
Recall that $d=\cos (\theta )T^{\ast }+\sin (\theta )\xi $ is the
decomposition of $d$ in its tangent and normal components.Let $\alpha $ be
the helix line of $M$ with unit speed. That is, $\frac{d\alpha }{ds}=T^{\ast
}$. Hence, doing the dot product with $\nabla f$ in each part of $d$ along
the helix lines of $M$, we obtain:%
\begin{equation*}
\left\langle \nabla f,d\right\rangle =\cos (\theta )\left\langle \nabla f,%
\frac{d\alpha }{ds}\right\rangle +\sin (\theta )\left\langle \nabla f,\xi
\right\rangle
\end{equation*}%
Due to the fact that $M$ is a $f$-eikonal helix submanifold, $\left\langle
\nabla f,d\right\rangle =$ constant along the helix lines of $M$. On the
other hand, $\left\langle \nabla f,\xi \right\rangle =0$ since $\nabla f\in
TM$. So, $\left\langle \nabla f,\frac{d\alpha }{ds}\right\rangle $ is
constant along the helix lines of $M$. It follows that the helix lines of $M$
are $f$-eikonal helix curves.
\end{proof}

\begin{theorem}
Let $i:M\rightarrow 
\mathbb{R}
^{n}$ be a submanifold and let $f:M\rightarrow 
\mathbb{R}
$ be a eikonal function, where $M$ has the induced metric by $%
\mathbb{R}
^{n}$. Let us assume that $\alpha :I\subset 
\mathbb{R}
\rightarrow M$ is a unit speed (parametrized by arc length function $s$)
curve on $M$ with unit tangent $T$ . Then,$\alpha $ is a $f$-eikonal helix
curve if and only if 
\begin{equation*}
\beta (s)=\phi (\alpha (s))=\left( i(\alpha (s)),f(\alpha (s))\right)
\subset 
\mathbb{R}
^{n}\times 
\mathbb{R}
\text{ }
\end{equation*}%
is a general helix with the axis $d=(0,1)$. Here, $\phi :M\rightarrow 
\mathbb{R}
^{n}\times 
\mathbb{R}
$ is given by $\phi (p)=\left( i(p),f(p)\right) $ and $i:M\rightarrow 
\mathbb{R}
^{n}$ is given by $i(p)=p$, where $p\in M$.
\end{theorem}

\begin{proof}
We consider the curve $\beta (s)=\left( i(\alpha (s)),f(\alpha (s))\right)
=\left( \alpha (s),f(\alpha (s))\right) $. Then, the tangent of $\beta $%
\begin{equation*}
\beta ^{%
{\acute{}}%
}(s)=\left( T,\frac{d(f\circ \alpha )}{ds}\right) \text{,}
\end{equation*}%
where $T$ is the unit tangent of $\alpha $.On the other hand, we know that $%
X[f]=\left\langle \nabla f,X\right\rangle $ for each $X\in TM$ (see
definition 2.4). In particular, for $X=T$,%
\begin{equation*}
T[f]=\left\langle \nabla f,T\right\rangle
\end{equation*}%
\begin{equation*}
\frac{d\alpha }{ds}[f]=\left\langle \nabla f,T\right\rangle
\end{equation*}%
and so, we have:%
\begin{equation*}
\frac{d(f\circ \alpha )}{ds}=\left\langle \nabla f,T\right\rangle \text{.}
\end{equation*}%
Therefore, we obtain%
\begin{equation}
\beta ^{%
{\acute{}}%
}(s)=\left( T,\left\langle \nabla f,T\right\rangle \right) \text{.}
\end{equation}%
Hence, doing the dot product with $d$ in each part of (3.1) , we get:%
\begin{equation}
\left\langle \beta ^{%
{\acute{}}%
}(s),d\right\rangle =\left\langle \nabla f,T\right\rangle \text{.}
\end{equation}%
From the equality (3.2), we can write%
\begin{equation*}
\left\Vert \beta ^{%
{\acute{}}%
}(s)\right\Vert .\cos (\theta )=\left\langle \nabla f,T\right\rangle \text{,}
\end{equation*}%
where $\theta $ is the angle between $d$ and $\beta ^{%
{\acute{}}%
}(s)$. It follows that%
\begin{equation}
\cos (\theta )=\frac{\left\langle \nabla f,T\right\rangle }{\sqrt{%
1+\left\langle \nabla f,T\right\rangle ^{2}}}\text{.}
\end{equation}%
If $\alpha $ is a $f$-eikonal helix curve,i.e. $\left\langle \nabla
f,T\right\rangle =$ constant, it can be easily seen that $\cos (\theta )=$%
constant by using (3.3). That is, $\beta $ is a general helix with the axis $%
d=(0,1)$.Conversely, we assume that $\beta $ is a general helix, i.e. $\cos
(\theta )=$constant.Hence, by using (3.3), we can write%
\begin{equation}
\left\langle \nabla f,T\right\rangle ^{2}=\frac{\cos ^{2}(\theta )}{\sin
^{2}(\theta )}=\text{constant }(\theta \neq 0)\text{.}
\end{equation}%
And so, from (3.4), we deduce that $\left\langle \nabla f,T\right\rangle =$%
constant. In other words, $\alpha $ is a $f$-eikonal helix curve.
\end{proof}

\begin{theorem}
Let $M\subset 
\mathbb{R}
^{n}$ be a complete connected smooth Riemannian submanifold without boundary
and let $M$ be isometric to a Riemannian product $N\times 
\mathbb{R}
$ . Let us assume that $f:M\rightarrow 
\mathbb{R}
$ be a non-trivial affine function (see main theorem in [9]). Then, all
geodesic curves on $M$ are $f$-eikonal helix curves.
\end{theorem}

\begin{proof}
Since $f:M\rightarrow 
\mathbb{R}
$ is a affine function, for each unit geodesic $\alpha :(-\infty ,\infty
)\rightarrow M$ there are constants $a$ and $b\in 
\mathbb{R}
$ such that%
\begin{equation*}
f\left( \alpha (s)\right) =as+b\text{.}
\end{equation*}%
for all $s\in (-\infty ,\infty )$ (see [8] or see [9]). On the other hand,
we know that 
\begin{equation*}
X[f]=\left\langle \nabla f,X\right\rangle
\end{equation*}%
for each $X\in TM$. In particular, for $X=T$ (the unit tangent of $\alpha $%
), 
\begin{equation*}
T[f]=\left\langle \nabla f,T\right\rangle
\end{equation*}%
\begin{equation*}
\frac{d\alpha }{ds}[f]=\left\langle \nabla f,T\right\rangle
\end{equation*}%
and so, we have 
\begin{equation}
\frac{d(f\circ \alpha )}{ds}=\left\langle \nabla f,T\right\rangle \text{.}
\end{equation}%
Moreover, since $f\left( \alpha (s)\right) =as+b$, $\dfrac{d(f\circ \alpha )%
}{ds}=$constant. Hence, from (3.5), we obtain 
\begin{equation*}
\left\langle \nabla f,T\right\rangle =\text{constant}
\end{equation*}%
along the curve $\alpha $.On the other hand, from Lemma 2.3 (see [11]), $%
\left\Vert \nabla f\right\Vert =$constant. Consequently, all geodesic curves
on $M$ are $f$-eikonal helix curves.
\end{proof}

\begin{example}
In Theorem 3.4., we take $M$ to be the cylindrical surface $S^{1}\times 
\mathbb{R}
$. And, we take $f$ to be the function%
\begin{equation*}
f:S^{1}\times 
\mathbb{R}
\rightarrow 
\mathbb{R}%
\end{equation*}%
\begin{equation*}
\left( x,t\right) \rightarrow f\left( x,t\right) =t\text{ .}
\end{equation*}%
Then, the curves $\alpha $ in the form%
\begin{equation*}
\alpha \left( s\right) =\left( \cos (c\frac{s}{\sqrt{c^{2}+a^{2}}}+d),\sin (c%
\frac{s}{\sqrt{c^{2}+a^{2}}}+d),a\frac{s}{\sqrt{c^{2}+a^{2}}}+b\right)
\end{equation*}%
$f$- eikonal helix curves since all geodesic curves on $M$ are the curves $%
\alpha $ with the unit tangent $T$, where $c^{2}+a^{2}$ $\neq 0$. Here, $a$,$%
b$,$c$,$d$ are real numbers.

In fact, $f\left( \alpha \left( s\right) \right) =a\dfrac{s}{\sqrt{%
c^{2}+a^{2}}}+b$ and $\dfrac{d\left( f\circ \alpha \right) }{ds}=$constant.
So, $\left\langle \nabla f,T\right\rangle =\dfrac{d(f\circ \alpha )}{ds}=$%
constant. On the other hand, $\left\Vert \nabla f\right\Vert =$constant
since $f$ is an affine function. Consequently, since $\left\Vert \nabla
f\right\Vert =$constant, $\left\Vert T\right\Vert =1$ and $\left\langle
\nabla f,T\right\rangle =$constant, the angle between $\nabla f$ and $T$
along the curves $\alpha $. In other words, the curves $\alpha $ are $f$%
-eikonal helix curves.
\end{example}

\section{THE AXIS OF $f$-EIKONAL HELIX CURVES}

In this seciton, in 3-dimensional Riemannian manifold, we find out the axis
of a $f$-eikonal helix curve and we give the relation between the curvatures
of the curve in Theorem 4.1. Then, we give more important corollary relating
to helix submanifolds

\begin{theorem}
Let $M\subset $ $%
\mathbb{R}
^{4}$ be a 3-dimensional Riemannian manifold and let $M$ be a complete
connected smooth without boundary. Also, let $M$ be isometric to a
Riemannian product $N\times 
\mathbb{R}
$ . Let us assume that $f:M\rightarrow 
\mathbb{R}
$ be a non-trivial affine function (see main theorem in [9]) and $\alpha
:I\rightarrow M$ be a $f$-eikonal helix curve. Then, the following
properties are hold:\bigskip

\textit{(1) The axis of }$\alpha $: 
\begin{equation*}
\nabla f=\left\Vert \nabla f\right\Vert \left( \cos (\theta )T+\sin (\theta
)B\right) \text{,}
\end{equation*}

where $\theta $ is constant.\medskip \bigskip

(2) $\dfrac{\tau }{\kappa }=$constant.
\end{theorem}

\ \ \ \ \ \ \ 

\begin{proof}
(1) Since $\alpha $ is $f$-eikonal helix curve, we can write%
\begin{equation}
\left\langle \nabla f,T\right\rangle =\text{constant.}
\end{equation}%
If we take the derivative in each part of (4.1) in the direction $T$ on $M$,
we have%
\begin{equation}
\left\langle \nabla _{T}\nabla f,T\right\rangle +\left\langle \nabla
f,\nabla _{T}T\right\rangle =0\text{.}
\end{equation}%
On the other hand, from Lemma 2.3 (see [11]), $\nabla f$ is parallel in $M$,
i.e. $\nabla _{X}\nabla f=0$ for arbitrary $X\in TM$. So, we get $\nabla
_{T}\nabla f=0$.Then, by using (4.2) and Frenet formulas, we obtain%
\begin{equation}
\kappa \left\langle \nabla f,N\right\rangle =0\text{.}
\end{equation}%
Since $\kappa $ is assumed to be positive, (4.3) implies that $\left\langle
\nabla f,N\right\rangle =0$. Hence, we can write the axis of $\alpha $ as 
\begin{equation}
\nabla f=\lambda _{1}T+\lambda _{2}B\text{.}
\end{equation}%
Doing the dot product with $T$ in each part of (4.4), we get 
\begin{equation}
\left\langle \nabla f,T\right\rangle =\lambda _{1}=\left\Vert \nabla
f\right\Vert \cos (\theta )\text{,}
\end{equation}%
where $\theta $ is the angle between $\nabla f$ and $T$. And, since $%
\left\Vert \nabla f\right\Vert ^{2}=\lambda _{1}^{2}+\lambda _{2}^{2}$, we
also have 
\begin{equation*}
\lambda _{2}=\left\Vert \nabla f\right\Vert \sin (\theta )
\end{equation*}%
by using (4.5).Finally, the axis of $\alpha $%
\begin{equation*}
\nabla f=\left\Vert \nabla f\right\Vert \left( \cos (\theta )T+\sin (\theta
)B\right) \text{.}
\end{equation*}%
(2) From the proof of (1), we can write%
\begin{equation}
\left\langle \nabla f,N\right\rangle =0.
\end{equation}%
If we take the derivative in each part of (4.6) in the direction $T$ on $M$,
we have%
\begin{equation}
\left\langle \nabla _{T}\nabla f,N\right\rangle +\left\langle \nabla
f,\nabla _{T}N\right\rangle =0\text{.}
\end{equation}%
And, from the proof of (1), $\nabla _{T}\nabla f=0$. Hence, from (4.7), 
\begin{equation}
\left\langle \nabla f,\nabla _{T}N\right\rangle =0\text{.}
\end{equation}%
By using Frenet formulas, from (4.8) we obtain%
\begin{equation}
-\kappa \left\langle \nabla f,T\right\rangle +\tau \left\langle \nabla
f,B\right\rangle =0\text{.}
\end{equation}%
On the other hand, by using (4.4), we can write as $\left\langle \nabla
f,T\right\rangle =\lambda _{1}$ and $\left\langle \nabla f,B\right\rangle
=\lambda _{2}$.Since $\lambda _{1}=\left\Vert \nabla f\right\Vert \cos
(\theta )$ and $\lambda _{2}=\left\Vert \nabla f\right\Vert \sin (\theta )$
from the proof of (1), we obtain%
\begin{equation}
\left\langle \nabla f,T\right\rangle =\left\Vert \nabla f\right\Vert \cos
(\theta )\text{ and }\left\langle \nabla f,B\right\rangle =\left\Vert \nabla
f\right\Vert \sin (\theta )\text{.}
\end{equation}%
So, by using (4.9) and the equalities (4.10), we have%
\begin{equation*}
\dfrac{\tau }{\kappa }=\cot (\theta )\text{=constant.}
\end{equation*}%
This completes the proof of the Theorem.
\end{proof}

The latter Theorem 4.1 has the following corollaries.

\begin{corollary}
Let $M\subset $ $%
\mathbb{R}
^{4}$ be a 3-dimensional Riemannian $f$-helix submanifold and let $M$ be a
complete connected smooth without boundary. Also, let $M$ be isometric to a
Riemannian product $N\times 
\mathbb{R}
$ . Let us assume that $f:M\rightarrow 
\mathbb{R}
$ be a non-trivial affine function (see main theorem in [9]).Then, $\dfrac{%
\tau }{\kappa }$ is constant along the helix lines of $M$.
\end{corollary}

\begin{proof}
It is obvious by using Theorem 4.1 and Theorem 3.2.
\end{proof}

\begin{corollary}
Let $M\subset $ $%
\mathbb{R}
^{4}$ be a 3-dimensional Riemannian $f$-helix submanifold and let $M$ be a
complete connected smooth without boundary. Also, let $M$ be isometric to a
Riemannian product $N\times 
\mathbb{R}
$ . Let us assume that $f:M\rightarrow 
\mathbb{R}
$ be a non-trivial affine function (see main theorem in [9]).Then,%
\begin{equation*}
\nabla f=\left\Vert \nabla f\right\Vert \left( \cos (\theta )T+\sin (\theta
)B\right)
\end{equation*}%
along the helix line of $M$. In other words, $\nabla f$ is the axis of the
helix line of $M$.
\end{corollary}

\begin{proof}
It is obvious by using Theorem 4.1 and Theorem 3.2.
\end{proof}

\section{THE RELATION BETWEEN $f$-EIKONAL HELIX CURVE AND GENERAL HELIX}

In section, we specify the relation between $f$-eikonal helix curve and
general helix.

\begin{lemma}
Let $M\subset 
\mathbb{R}
^{n}$ be a Riemannian helix submanifold with respect to the unit direction $%
d\in 
\mathbb{R}
^{n}$ and $f:M\rightarrow 
\mathbb{R}
$ be a function. Let us assume that $\alpha :I\subset 
\mathbb{R}
\rightarrow M$ is a unit speed (parametrized by arc length function $s$)
curve on $M$ with unit tangent $T$ . Then,the normal component $\xi $ of $d$
is parallel normal in the direction $T$ if and only if $\ (\nabla f)^{%
{\acute{}}%
}\in $ $TM$ along the curve $\alpha $, where $T^{\ast }=\nabla f$ is the
unit tangent component of the direction $d$.
\end{lemma}

\begin{proof}
We assume that the normal component $\xi $ of $d$ is parallel normal in the
direction $T$. Since $T$ and $\nabla f$ $\in TM$, from the Gauss equation in
Definition 2.3,%
\begin{equation}
D_{T}\nabla f=\nabla _{T}\nabla f+V(T,\nabla f)
\end{equation}%
According to this Lemma, since the normal component $\xi $ of $d$ is
parallel normal in the direction $T$, i.e.$\nabla _{T}^{\bot }\xi =0$ (see
Remark 2.1), from (2.2) in Remark 2.2 ($0<\theta <\frac{\pi }{2}$) 
\begin{equation}
V(T,\nabla f)=0
\end{equation}%
So, by using (5.1),(5.2) and Frenet formulas, we have:%
\begin{equation*}
D_{T}\nabla f=\frac{d\nabla f}{ds}=(\nabla f)^{%
{\acute{}}%
}=\nabla _{T}\nabla f\text{.}
\end{equation*}%
That is, the vector field $(\nabla f)^{%
{\acute{}}%
}\in TM$ along the curve $\alpha $, where $TM$ is the tangent space of $M$.

Conversely, let us assume that $(\nabla f)^{%
{\acute{}}%
}\in $ $TM$ along the curve $\alpha $. Then, from Gauss equation, $%
V(T,\nabla f)=0$. Hence, from (2.2) in Remark 2.2 ($0<\theta <\frac{\pi }{2}$%
), $\nabla _{T}^{\bot }\xi =0$ . That is, the normal component $\xi $ of $d$
is parallel normal in the direction $T$. This completes the proof.
\end{proof}

\begin{lemma}
All $f$-eikonal helix curves with the constant axis $\nabla f$ are general
helices.
\end{lemma}

\begin{proof}
For all $f$-eikonal helix curves, we know that the angle between $\nabla f$
and unit tangent vector fields $T$ of these curves is constant along these
curves. Moreover, from this lemma, $\nabla f$ is constant. Therefore,
tangent vector fields of these curves make a constant angle with the
constant direction $\nabla f$. It follows that all $f$-eikonal helix curves
with the constant axis $\nabla f$ are general helices by using definition
2.6.
\end{proof}

\begin{theorem}
Let $M\subset 
\mathbb{R}
^{n}$ be a complete connected smooth Riemannian helix submanifold with
respect to the unit direction $d\in 
\mathbb{R}
^{n}$ and $f:M\rightarrow 
\mathbb{R}
$ be an affine function . Let us assume that $\alpha =\alpha \left( s\right)
:I\subset 
\mathbb{R}
\rightarrow M$ is a $f$-eikonal helix curve on $M$ with unit tangent $T$ .
Then, if the normal component $\xi $ of $d$ is parallel normal in the
direction $T$, then the $f$-eikonal helix curve $\alpha $ is a general
helix, where $T^{\ast }=\nabla f$ is the unit tangent component of the
direction $d$.
\end{theorem}

\begin{proof}
Since $T$ and $\nabla f$ $\in TM$, from the Gauss equation in Definition 2.3,%
\begin{equation}
D_{T}\nabla f=\nabla _{T}\nabla f+V(T,\nabla f)
\end{equation}%
Since $f$ is an affine function ,$\nabla f$ \ parallel in $M$, i.e. $\nabla
_{X}\nabla f=0$ for arbitrary $X\in TM$, $\nabla _{T}\nabla f=0$ (see Lemma
2.3. in [11]). On the other hand, according to the Lemma 5.1, $(\nabla f)^{%
{\acute{}}%
}\in $ $TM$ due to the fact that the normal component $\xi $ of $d$ is
parallel normal in the direction $T$. Therefore, from Gauss equation, $%
V(T,\nabla f)=0$. Hence, from (5.3), we have:%
\begin{equation*}
D_{T}\nabla f=\frac{d\nabla f}{ds}=(\nabla f)^{%
{\acute{}}%
}=0
\end{equation*}%
along the curve $\alpha $. That is, the axis $\nabla f$ of $\alpha $ is
constant. So, the above Lemma 5.2 follows that the $f$-eikonal helix curve $%
\alpha $ is a general helix. This completes the proof.
\end{proof}

\end{document}